\documentclass{amsart}

\usepackage{latexsym}
\usepackage{amssymb}
\usepackage{amsmath}
\usepackage{amsthm}
\usepackage{amsfonts}
\usepackage{color}

\usepackage{pictexwd,dcpic}

\usepackage{graphicx}

\usepackage{psfrag}
\usepackage{hyperref}
\usepackage{comment}
\usepackage{multirow} 

\newtheorem{thm}{Theorem}[section]
\newtheorem{cor}[thm]{Corollary}
\newtheorem{lem}[thm]{Lemma}
\newtheorem{prop}[thm]{Proposition}

\newtheorem{theorem}{Theorem}

\theoremstyle{definition}
\newtheorem{defn}[thm]{Definition}

\newtheorem{example}{Example}
\newtheorem{rem}[thm]{Remark}

\newtheorem*{ack}{Acknowledgements}

\numberwithin{equation}{section}

\newcommand{\scal}[1]{\langle #1 \rangle}

\DeclareMathOperator{\Sym}{Sym}

\newcommand{\RR}{\mathbb{R}}
\newcommand{\CC}{\mathbb{C}}

\newcommand{\FF}{\mathbb{F}}

\DeclareMathOperator{\codim}{codim}

\DeclareMathOperator{\Tub}{Tub}

\newcommand{\sphere}{\mathrm{\mathbb{S}}}

\newcommand{\ZZ}{\mathbb{Z}}

\newcommand{\indecomposable}{indecomposable}
\newcommand{\decomposable}{decomposable}

\newcommand{\fol}{\mathcal{F}}
\newcommand{\SO}{\mathsf{SO}}
\newcommand{\Spin}{\mathrm{Spin}}
\newcommand{\On}{\mathrm{O}}
\newcommand{\SU}{\mathsf{SU}}
\newcommand{\U}{\mathsf{U}}
\newcommand{\Sp}{\mathsf{Sp}}
\DeclareMathOperator{\End}{End}

\newcommand{\Fix}{\mathrm{Fix}}

\newcommand{\GL}{\mathrm{GL}}

\newcommand{\ul}{\underline}

\newcommand{\lra}{\longrightarrow}
\newcommand{\lmt}{\longmapsto}
\newcommand{\ra}{\rightarrow}

\newcommand{\In}{\subseteq}

\newcommand{\tr}{\textrm{tr}}

\newcommand{\DD}{\mathbb{D}}
\newcommand{\HH}{\mathbb{H}}

\newcommand{\PP}{\Bbb{P}}

\newcommand{\mmod}{\textrm{mod\,}}

\begin{document}



\title[Clifford algebras and new foliations]{Clifford algebras and new singular Riemannian foliations in spheres}


\author[M. Radeschi]{Marco Radeschi}
\address{Mathematisches Institut, WWU M\"unster, Germany.}
\email{mrade\_02@uni-muenster.de}
\thanks{}

\dedicatory{Dedicated to the memory of Sergio Console}

\date{\today}




\begin{abstract}
Using representations of Clifford algebras we construct \indecomposable{} singular Riemannian foliations on round spheres, most of which are non-homogeneous. This generalises the construction of non-homogeneous isoparametric hypersurfaces due to by Ferus, Karcher and M\"unzner.
 \end{abstract}

\maketitle




A singular Riemannian foliation on a Riemannian manifold $M$ is, roughly speaking, a partition of $M$ into connected complete submanifold, not necessarily of the same dimension, that locally stay at a constant distance from each other. Singular Riemannian foliations on round spheres provide local models of general singular Riemannian foliations around a point.

An example of singular Riemannian foliation on round spheres is given by the decomposition into the orbits of an isometric group action, and such a foliation is called \emph{homogeneous}.

A different family of singular Riemannian foliations on spheres is induced by isoparametric hypersurfaces. A hypersurfaces of $\sphere^n$ is called isoparametric if it has constant principal curvatures. Isoparametric hypersurfaces were first studied by Cartan who classified those with $g\leq 3$ distinct principal curvatures, and a lot of progress has been made (cf. for example the surveys \cite{Cec, Tho}), even though the complete classification is still an important open problem. Every isoparametric hypersurface partitions the sphere into parallel hypersurfaces, which are isoparametric as well, and this partition is a special example of a singular Riemannian foliation. For a long time all the known codimension 1 singular Riemannian foliations from isoparametric hypersurfaces appeared to be orbits of some isometric group action on $\sphere^n$, so much so that Cartan asked \cite{Car} whether every isoparametric hypersurface arised in this way. The question was answered in the negative by Ozeki and Takeuchi \cite{OTI, OTII}, who found infinite families of non homogeneous isoparametric foliations with $4$ distinct principal curvatures defined in terms of representations of Clifford algebras. These examples were then extended to a larger class of nonhomogeneous isoparametric foliations by Ferus, Karcher and M\"unzner \cite{FKM}, again using Clifford algebras. We call these axemples the \emph{FKM examples}. It has been proven that every foliation in round spheres by isoparametric hypersurfaces with 4 principal curvatures is either homogeneous or of FKM type, except possibly for a finite number of isolated cases (cf. \cite{Imm}). 

As in the isoparametric case, classifying non-homogeneous singular Riemannian foliations seems a very complex problem. A trivial way to obtain new foliations from old ones is called \emph{spherical join}. Given singular Riemannian foliations $(\sphere^{n_i},\fol_i)$, $i=1,2$, the spherical join gives a new foliation $(\sphere^{n_1+n_2+1},\fol_1\star \fol_2)$. Any foliation that cannot be written as a spherical join is called \emph{\indecomposable}, and every foliation can be written in an essentially unique way as a spherical join of \indecomposable{} ones. Because of this, our main interest lies in finding non-homogeneous, \indecomposable{} singular Riemannian foliations.

The only known \indecomposable{} non-homogeneous singular Riemannian foliation, other than the FKM examples mentioned above, is the foliation in $\sphere^{15}$ given by the fibers of the Hopf fibration $\sphere^{15}\to \sphere^8$. Recently A. Lytchak and B. Wilking proved, using a previous result of Wilking \cite{Wil} and Grove-Gromoll \cite{GG}, that this is the only non-homogeneous \emph{regular} foliation, i.e., with leaves of the same dimension \cite{LW}.

In this paper, as in \cite{FKM}, we use Clifford systems to produce a large class of \indecomposable{}, non-homogeneous singular Riemannian foliations of arbitrary codimension, which in particular include all the previously known examples. Before we state the result, recall that a Clifford system can be thought of as a family $C=(P_0, \ldots P_m)$ of symmetric matrices in $(\RR^{2l},\scal{\,,\,})$ such that $P_i^2=Id$ for all $i=0,\ldots m$ and $P_iP_j=-P_jP_i$ for $i\neq j$. We define the map
\begin{align*}
\pi_C: \sphere^{2l-1}&\lra \RR^{m+1}\\
x&\lmt \Big(\scal{P_0x,x},\ldots \scal{P_mx,x}\Big).
\end{align*}


\begin{theorem}\label{T:CliffordFoliations}
Let $C=(P_0,\ldots P_m)$ be a Clifford system on $\RR^{2l}$. Then the image of $\pi_C$ is contained in the unit disk $\DD_C$ around the origin in $\RR^{m+1}$, and the following hold:
\begin{enumerate}
\item The preimages of $\pi_C$ are connected if $l\neq m+1$ and in this case they define a singular Riemannian foliation $(\sphere^{2l-1},\fol_C)$ whose leaf space is either the $m$-sphere $\sphere_C=\partial \DD_C$ (if $l=m$) or the disk $\DD_C$ (if $l>m+1$). In either case the induced metric on the quotient is a round metric of constant sectional curvature $4$.
 \item The foliation $(\sphere^{2l-1},\fol_C)$ is homogeneous if and only if $m=1,2$ or $m=4$ and $P_0\cdot P_1\cdot P_2\cdot P_3\cdot P_4=\pm Id$, in which cases it is spanned by the orbits of the diagonal action of $\SO(k)$ on $\RR^k\times \RR^k$ ($m=1$), $\SU(k)$ on $\CC^k\times \CC^k$ ($m=2$) or $\Sp(k)$ on $\HH^k\times \HH^k$ ($m$=4).
\end{enumerate}
\end{theorem}

When the leaf space is a sphere one recovers the Hopf fibrations $\pi_C:\sphere^{2m-1}\to \sphere^m$, $m=2,4,8$. When the leaf space is $\DD_C$ with the round metric (also \emph{hemisphere metric}) the $\pi_C$-preimages in $\sphere^{2l-1}$ of the concentric spheres in $\DD_C$ give rise to the FKM family associated to the Clifford system $C$.
\\

A singular Riemannian foliation $\fol_0$ on the $m$-sphere $\sphere_C=\partial \DD_C\In \RR^{m+1}$ extends by homotheties to a singular Riemannian foliation $\fol_0^h$ on $\DD_C$ (with the hemisphere metric) and the $\pi_C$-preimages of the leaves of $\fol_0^h$ define a new foliation $\fol_0\circ\fol_C$. This is a special case of a more general construction of Lytchak \cite[Sect. 2.5]{Lyt}.

\begin{theorem}\label{T:ComposedFoliations}
Let $C$ a Clifford system on $\RR^{2l}$ and let $(\sphere^{2l-1},\fol_C)$ be the associated Clifford foliation.
\begin{enumerate}
\item If $\fol_0$ is any singular Riemannian foliation on $\sphere_C$, then the foliation $(\sphere^{2l-1},\fol_0\circ\fol_C)$ is a singular Riemannian foliation as well.
\item Let $C_{8,1}$ and $C_{9,1}$ denote, respectively, the unique Clifford systems $(P_0,\ldots P_8)$ on $\RR^{16}$ and $(P_0,\ldots P_9)$ on $\RR^{32}$. If $C\neq C_{8,1},\,C_{9,1}$ then $(\sphere^{2l-1},\fol_0\circ \fol_C)$ is homogeneous if and only if both $\fol_0$ and $\fol_C$ are homogeneous. If $C=C_{9,1}$ and $(\sphere^{31},\fol_0\circ \fol_C)$ is homogeneous, then $\fol_0$ is homogeneous.
\end{enumerate}
\end{theorem}

Statement (2) of Theorem \ref{T:ComposedFoliations} fails in the case of $C=C_{8,1}$ or $C=C_{9,1}$, as there are examples of (homogeneous) foliations $(\sphere^m,\fol_0)$ such that $\fol_0\circ\fol_C$ is homogeneous, while $C$ itself is not. It would be interesting to have a complete characterization of the homogeneous foliations of type $\fol_0\circ \fol_C$ in these last two cases.

We call the foliations $\fol_C$ described above \emph{Clifford foliations}, and the foliations $\fol_0\circ\fol_C$ \emph{composed foliations}. Notice that in a Clifford foliation the set of singular leaves is a connected, smooth, non totally geodesic submanifold of $\sphere^{2l-1}$. This can never happen for \decomposable{} foliations, and therefore every Clifford foliation is \indecomposable.
\begin{example}
If $\fol_0$ is a trivial foliation whose leaves consist of points, $\fol_0\circ \fol_C=\fol_C$ and in particular every Clifford foliation is a composed foliation as well. If $\fol_0$ is the trivial foliation consisting of one leaf, $\fol_0\circ \fol_C$ is the codimension $1$ FKM examples corresponding to the Clifford system $C$. Since the foliation induced by the Hopf fibration $\sphere^{15}\to \sphere^8$ is of the form $\fol_C$, all previously known examples of \indecomposable, non-homogeneous foliations are of the form $\fol_0\circ \fol_C$, with $\fol_0$ trivial.
\end{example}
\begin{example}
Let $(\sphere^{15},\fol_C)$ be the Clifford foliation with quotient $\sphere^{8}\In \RR^9$. The group $\SO(3)\times \SO(3)$ acts on $\RR^9=\RR^3\otimes \RR^3$ via the tensor product representation, and the restriction of this action on the unit sphere induces a (homogeneous) foliation $(\sphere^8,\fol_0)$ whose quotient space is a spherical triangle of curvature $1$ with angles $\pi/3, \pi/3, \pi/2$. The composed foliation $\fol_0\circ \fol_C$ is thus a singular Riemannian foliation on $\sphere^{15}$ whose quotient is isometric to a spherical triangle of curvature $4$, with angles $\pi/3, \pi/3, \pi/2$. Such a quotient does not appear as a quotient of an isometric group action of cohomogeneity $2$ (see Straume classification \cite[Table II]{Str}) and therefore $\fol_0\circ \fol_C$ is non-homogeneous. Moreover, such a triangle does not admit submetries onto a segment and therefore $\fol_0\circ\fol_C$ is not contained in any codimension $1$ foliation. To our knowledge, this is the only known singular Riemannian foliation of codimension $2$ with this property. In the homogeneous case, by contrast, it is known that every action of cohomogeneity 2 on a round sphere is contained in a larger action of cohomogeneity 1 (see for example \cite[Theorem 1.1]{GL2}).
\end{example}
\begin{example}
Let $C$ be a Clifford system on $\RR^{2l}$, and $(\sphere_C,\fol_0)$ be a singular Riemannian foliation without $0$-dimensional leaves. Then the leaf space $\sphere_C/\fol_0$ has diameter $\leq {\pi\over 2}$, and the composed foliation $(\sphere^{2l-1},\fol_0\circ \fol_C)$ has quotient of diameter $\pi/4$, which in particular is strictly smaller than $\pi/2$. Such foliations are called \emph{irreducible}, because in the homogeneous setting the representations with quotient of diameter $<\pi/2$ are precisely the irreducible ones. Since \decomposable{} foliations have quotient with diameter $\geq\pi/2$, it follows in particular that these examples are \indecomposable.
\end{example}

Unlike the FKM examples, inequivalent Clifford system give rise to different Clifford foliations, see Proposition \ref{P:distinguish}. Moreover, Clifford foliations can be geometrically characterized as the only singular Riemannian foliations on spheres whose quotient is a sphere or a hemisphere of curvature 4. More precisely, let $\mathcal{G}$ the class of singular Riemannian foliations on a round sphere, whose quotient is a sphere or a hemisphere of curvature 4 and let $\mathcal{A}$ be the class of Clifford systems. Then the following holds.
\begin{theorem}\label{T:rigidity}
The assignment $C\mapsto \fol_C$ determines a bijection
\[
\mathcal{A}/\{\text{geometric equivalence}\}\stackrel{\simeq}{\lra} \mathcal{G}/\{\text{congruence}\}
\]
\end{theorem}
This is somewhat surprising, since it establishes an equivalence between purely algebraic  and purely geometric objects.

The paper is structured as follows. After preliminary Section \ref{s:preliminaries} we provide the construction of the Clifford foliations in Section \ref{s:construction} and that of composed foliations in Section \ref{s:composed}. In Section \ref{s:composed} we also prove that both Clifford and composed foliations are singular Riemannian foliations, thereby finishing the proofs of the first  statements of Theorems \ref{T:CliffordFoliations} and \ref{T:ComposedFoliations}. In Sections  \ref{s:rigidity} we prove Theorem \ref{T:rigidity} and finally in Section \ref{s:homogeneous} we prove the homogeneity statements in Theorems  \ref{T:CliffordFoliations} and \ref{T:ComposedFoliations}. In Section \ref{S:CROSS} we extend the construction of Clifford and composed foliations to the case of complex and quaternionic projective spaces, and the Cayley plane. The last Section \ref{s:properties} is devoted to pointing out some properties that make Clifford and composed foliations very different from homogeneous ones. With the exception of the Hopf fibration $\sphere^{15}\to \sphere^8$, the quotient of every other previously known \indecomposable{} foliation was isometric to the orbit space of a group action, and therefore shared many properties of homogeneous foliations. The main goal of this last section is thus to provide evidence that homogeneous foliations are indeed very special.

\begin{ack}
The author thanks Alexander Lytchak for many helpful conversations and for inspiring Section \ref{s:properties}, as well as Karsten Grove, Wolfgang Ziller and Marcos Alexandrino for their interest and many comments on a preliminary version of this work. The author also thanks Luigi Vezzoni and the whole Mathematics Department of Universit\`a di Torino for the hospitality during his visit, where part of this work was produced.
\end{ack}
%
%

\section{Preliminaries}\label{s:preliminaries}
\subsection{Singular Riemannian foliations}
\begin{defn}
Let $M$ be a Riemannian manifold, and $\fol$ a partition of $M$ into complete, connected, injectively immersed submanifolds, called \emph{leaves}. The pair $(M,\fol)$ is called:
\begin{itemize}
\item a \emph{singular foliation} if there is a family of smooth vector fields $\{X_i\}$ that span the tangent space of the leaves at each point.
\item a \emph{transnormal system} if any geodesic starting perpendicular to a leaf, stays perpendicular to all the leaves it meets. Such geodesics are called \emph{horizontal geodesics}.
\item a \emph{singular Riemannian foliation} if it is both a singular foliation and a transnormal system.
\end{itemize}
\end{defn}

Given a singular foliation $(M,\fol)$, the  \emph{space of leaves}, denoted by $M/\fol$, is the set of leaves of $\fol$ endowed with the topology induced by the \emph{canonical projection} $\pi:M\to M/\fol$ that sends a point $p\in M$ to the leaf $L_p\in \fol$ containing it.

On a singular Riemannian foliation $(M,\fol)$ it is possible to define a stratification, as follows. For each nonnegative integer $r$ define $\Sigma_r$ to be the union of leaves of dimension $r$. The connected components of each $\Sigma_r$ are (possibly noncomplete) submanifolds, and such connected components are called \emph{strata} of $(M,\fol)$. We denote by $\dim\fol$ the maximum dimension of the leaves in $\fol$, and call \emph{regular leaf} a leaf of maximal dimension and \emph{regular point} a point in a regular leaf. The set of regular leaves $\Sigma_{\dim \fol}$ is open, dense and connected, and therefore it defines a stratum which we call the \emph{regular stratum}.

A singular Riemannian foliation $(M,\fol)$ is called \emph{closed} if all the leaves of $\fol$ are closed. If $(M,\fol)$ is a closed foliation then all the leaves are at a constant distance from each other, and the space of leaves $M/\fol$ has the structure of a Hausdorff metric space. Moreover, the strata $\Sigma$ project to orbifolds in $M/\fol$, and the restriction of $\pi:M\to M/\fol$ to $\Sigma$ is a Riemannian submersion. In particular, $M/\fol$ is stratified by orbifolds $\Sigma/\fol$, and the \emph{regular stratum} $\Sigma_{\dim \fol}/\fol$ is open and dense in $M/\fol$.

A typical example of singular Riemannian foliation is provided by the orbit decomposition of a Riemannian manifold $M$ into the orbits of an isometric actions of a connected Lie group. Such foliations are called \emph{homogeneous}.

Finally, we define two singular Riemannian foliations $(M,\fol)$, $(M',\fol')$ \emph{congruent} if there is an isometry of $M\to M'$ that takes leaves of $\fol$ isometrically onto leaves of $\fol'$.
\subsection{Clifford algebras and Clifford systems}\label{S:Clifford-syst}
In this section we recall the basic definitions and results on Clifford algebras and Clifford systems, which we will need later on, see reference \cite[Section 3]{FKM}.

The \emph{Clifford algebra} $C\ell_{m}(\RR)=C\ell(V)$ is constructed from a (real) vector space $V$ of dimension $m$ with a positive definite inner product $\scal{,}$ and is defined by the quotient of the tensor algebra $T(V)$ by the ideal $x\otimes y+y\otimes x-2\scal{x,y}1$, where $1$ is the unit element in $T(V)$. The vector space $V$ naturally embeds in $C\ell(V)$, and every $x,y\in V$ satisfy the relation
\[
x\cdot y + y\cdot x= 2\scal{x,y} 1.
\]
A \emph{representation} of a Clifford algebra $C\ell_{m}(\RR)$, or \emph{Clifford module}, is an algebra homomorphism $\rho:C\ell_{m}(\RR)\to \End(\RR^{n})$. Two representations $\rho, \rho'$ are said to be \emph{equivalent} if there is an isomorphism $A\in \GL(\RR^n)$ such that $\rho'=A^{-1}\circ \rho \circ A$.
The restriction
\[
\rho|_V:V\to\End(\RR^n)
\]
will be called \emph{Clifford system} on $\RR^n$, and denoted by $C$. We will also denote by $\RR_C$ the image $\rho(V)$, and call $\dim(V)$ the \emph{rank} of $C$. Given an orthonormal basis $x_0,\ldots x_{m-1}$ of $V$, the images $P_i=\rho(x_i)\in \RR_C$ are matrices that satisfy the relations $P_i^2=Id$ and $P_iP_j=-P_jP_i$ for $i\neq j$. 
For every Clifford system $C$ it is possible to find an inner product $\scal{\,,\,}$ on $\RR^n$ such that $\RR_C$ consists of symmetric matrices, and from now on we will fix one such inner product. If one endows $\Sym^2(\RR^n)$ with the inner product $\scal{A,B}={1\over n}\tr(AB)$, the map $C:V\to R_C\In \Sym^2(\RR^n)$ is an isometry, i.e., $\scal{C(x),C(y)}=\scal{x,y}$.

Let $\sphere_C$ denote the unit sphere in $\RR_C$. For any $P\in \sphere_C$, $P^2=Id$ and therefore $P$ has eigenvalues $\pm 1$, with eigenspaces $E_{\pm}( P)$. If $Q\perp P$, $PQ=-QP$ and therefore $Q$ takes the \emph{positive eigenspace} $E_+( P)$ isomorphically into the \emph{negative eigenspace} $E_-( P)$, and vice versa. In particular, $\dim E_+( P)=\dim E_-( P)$ and since $\RR^n$ splits as a sum $E_+( P)\oplus E_-( P)$, $n$ is always even dimensional, and we will write $n=2l$.
\\

Given two Clifford systems $C: V\to \Sym^2(\RR^{2l})$, $C':V\to \Sym^2(\RR^{2r})$ on the same Clifford algebra $C\ell(V)$, one can produce a new Clifford system $C\oplus C': V\to \Sym^2(\RR^{2(l+r)})$ by letting $(C\oplus C')(x)=(C(x),C'(x))$. We call $C\oplus C'$ a \emph{reducible} Clifford system. Any Clifford system that cannot be written as a non trivial sum is called \emph{irreducible}. If  $C$ is an irreducible Clifford system of rank $m+1$ on $\RR^{2l}$ then $l=\delta(m)$, where the function $\delta(m)$ is given as follows
\begin{equation}\label{E:table-delta-m}
\begin{array}{c||c|c|c|c|c|c|c|c|c}m & 1 & 2 & 3 & 4 & 5 & 6 & 7 & 8 & 8+n \\\hline \delta(m) & 1 & 2 & 4 & 4 & 8 & 8 & 8 & 8 & 16\delta(n)\end{array}
\end{equation}

Two Clifford systems $C,C':V\to \Sym^2(\RR^n)$ are \emph{algebraically equivalent} if there is an isometry $A\in \On(\RR^n)$ such that $C'=A^{-1}\circ C \circ A$, and \emph{geometrically equivalent} if there is an isometry $A\in \On(\RR^n)$ such that $\RR_{C'}=\RR_{A^{-1}\circ C \circ A}$. If $m\not\equiv 0 (\mmod 4)$ there is a unique irreducible Clifford system on $\RR^n$ up to algebraic equivalence, and in particular geometric equivalence. For $m\equiv 0 (\mmod 4)$, there are two algebraic equivalence classes of Clifford systems, such that if $(P_0,P_1,\ldots P_m)$ is the basis for one such class then the other can be identified with $(-P_0,P_1,\ldots, P_m)$. In particular, there is one geometric class of irreducible Clifford systems for $m\equiv 0 (\mmod{4})$ as well.

Any Clifford system is algebraically equivalent to a direct sum of irreducible ones. In particular if $C$ is a Clifford system of rank $m+1$ on $\RR^{2l}$ then $l=k\delta(m)$ for some $k>0$ and if $m\not\equiv 0 (\mmod 4)$ there is only one algebraic equivalence of Clifford systems for each $k$. For $m\equiv 0 (\mmod 4)$, however, a Clifford system of rank $m+1$ on $\RR^{2l}$, $l=k\delta(m)$ can be obtained by taking combinations of the two algebraically distinct irreducible Clifford systems, resulting in $\lfloor{k\over 2}\rfloor+1$ geometrically distinct Clifford systems. These can be told apart by the invariant $|\tr(P_0\cdot P_1\cdot\ldots \cdot P_m)|$, where $(P_0,\ldots P_m)$ is a basis of $C$, which takes exactly the $\lfloor{k\over 2}\rfloor +1$ dinstinct values $k-2j$, $2j\leq k$.

We will use the notation $C_{m,k}$ to denote a Clifford system of rank $m+1$ on $\RR^{2k\delta(m)}$. By the discussion above, when $m\not\equiv 0 (\mmod 4)$ or $k=1$ the notation $C_{m,k}$ uniquely determines the Clifford system up to geometric equivalence.

Finally, we recall that if $C$ is a Clifford system and $P,Q$ are elements in $\RR_C$, then $\scal{Px,Qx}=\scal{P,Q}\|x\|^2$.

\subsection{The construction of the FKM examples}\label{constructionFKM}
In \cite{FKM}, the authors use Clifford system to produce new examples of isoparametric hypersurfaces in spheres with $4$ principal curvatures. In the following we will refer to them as the \emph{FKM examples}. Given a Clifford system $C$ of rank $m+1$ on $\RR^{2l}$, $l=k\delta(m)$ and fixing a basis $P_0,\ldots P_m$ of $C$, they define a polynomial $F:\RR^{2l}\to \RR$ by
\[
F(x)=\scal{x,x}-2\sum_{i=0}^m\scal{P_ix,x}^2
\]
This polynomial restricts to a map $F_0:\sphere^{2l-1}\to [-1,1]$, such that the preimages of the level sets are smooth, closed submanifolds of $\sphere^{2l-1}$. These submanifolds depend on the values of $m$ and $l$.
\begin{itemize}
\item If $l>m+1$, $F_0$ is surjective, and the level sets are connected. The regular level sets form a family of isoparametric submanifolds, while the preimages $M_{\pm}=F_0^{-1}(\pm 1)$ are the focal submanifolds.
\item If $l=m+1$ (which can only happen for $(m,k)\in\{(1,2), (3,1), (7,1)\}$) then $F_0$ is still surjective, but the fibers of $F_0$ are disconnected except for $M_-$, which is a hypersurface.
\item If $l=m$ (which can only happen for $(m,k)\in\{(2,1), (4,1), (8,1)\}$) then $F_0\equiv -1$, and $M_-=\sphere^{2l-1}$.
\end{itemize}
The map $F_0$ restricts to a submersion in the regular part $\sphere^{2l-1}\setminus(M_+\cup M_-)$. This map is not a Riemannian submersion, nevertheless it can be modified to become one, and its qutiotient is an interval of length $\pi/4$. 

Restricting to the generic case $l>m+1$, most of these examples are non homogeneous. More specifically, given a Clifford system $C$ of rank $m+1$ on $\RR^{2l}$, $l=k\delta(m)$, the corresponding FKM example is homogeneous only for the following values of $(m,k)$ (cf. \cite[Section 4.4]{FKM}, \cite[Table F]{GWZ}):
\begin{equation}\label{Table homog}
\begin{array}{c||c|c|c|c}(m,k) & (1,k) & (2,k) & (4,k) & (9,1) \\ \hline condition & k\geq 2 & k\geq 1 & k\geq 1,\;P_0P_1P_2P_3 P_4=\pm Id & - \end{array}
\end{equation}
Where $(P_0,\ldots, P_4)$ is a basis of $C$.

%
%

\section{The construction}\label{s:construction}

We now proceed to define the new examples of singular Riemannian foliations of higher codimension. Let $C$ be a Clifford system of rank $m+1$ on $\RR^{2l}$, $l=k\delta(m)$. On the unit sphere $\sphere^{2l-1}\In \RR^{2l}$ (endowed with the canonical inner product which we also denote by $\scal{\cdot,\cdot}$), consider the function
\begin{align*}
\pi_C:\sphere^{2l-1}&\lra \RR_C=\RR^{m+1}
\end{align*}

that takes $x\in \sphere^{2l-1}$ to the unique element $\pi_C(x)\in \RR_C$ defined by the property
\begin{equation}\label{property}
\scal{\pi_C(x), P}=\scal{Px,x}\qquad \forall P\in \RR_C
\end{equation}
Fixing an orthonormal basis $(P_0,\ldots, P_m)$ of $\RR_C$, the map $\pi_C$ can be rewritten as
\[
\pi_C(x)=\Big(\scal{P_0x,x},\ldots \scal{P_mx,x}\Big).
\]
\begin{lem}
The image of $\pi_C$ is contained in the unit disk $\DD_C$ of $\RR_C$.
\end{lem}
\begin{proof}
Let $x_0\in \sphere^{2l-1}$ and $P=\pi_C(x_0)$. It is enough to show that $\|P\|\leq 1$. By the defining equation \eqref{property} we have
\begin{equation}\label{inequality}
\|P\|^2=\scal{P,P}=\scal{Px_0,x_0}\leq \|P\|\cdot \|x_0\|^2=\|P\|
\end{equation}
Hence $\|P\|\leq 1$ as we wanted.
\end{proof}
The relation with the polynomial $F_0$  of Ferus, Karcher, Munzner is explicit, as $F_0$ factors through $\pi_C$ as $F_0=f\circ \pi_C$, where $f:\RR_C\to \RR$ is the polynomial
\begin{equation}\label{E:MyExaToFKM}
f\left(P\right)=1-2\|P\|^2.
\end{equation}
We endow $\DD_C$ with a hemisphere metric of constant sectional curvature $4$, so that the boundary $\sphere_C=\partial \DD_C$ is totally geodesic. From now on, we will always assume that the metric on $\DD_C$ is the round one. 
\\

\begin{rem}\label{R:M+}
By equation \eqref{E:MyExaToFKM}, the preimages under $\pi_C$ of the concentric spheres in $\DD_C$ give back the FKM family associated to the Clifford system $C$. In particular the preimage of the origin is the focal manifold $M_+$ and the preimage of the boundary is $M_-$.
\end{rem}

\begin{rem}  If $C$ and $C'$ are algebraically equivalent Clifford systems, by definition there exists an orthogonal map $A\in O(\RR^{2l})$ such that $\pi_{C'}=\pi_C\circ A$. In particular, up to orthogonal transformation $\pi_C$ only depends on the algebraic equivalence class of $C$. We will see in Section \ref{P:distinguish} that the converse is also true, namely the geometric equivalence class of $C$ is uniquely determined by $\pi_C$.
\end{rem}
\begin{prop}
Given a Clifford system $C$ of rank $m+1$ on $\RR^{2l}$, $l=k\delta(m)$, the corresponding map
\begin{equation}
\pi_C:\sphere^{2l-1}\to \DD_C
\end{equation}
satisfies:
\begin{enumerate}
\item\label{1} The preimage of $P\in \sphere_C=\partial \DD_C$ is the unit sphere $E_+^1(P )$ in the positive eigenspace $E_+( P)$. Moreover, the restriction $\pi|_{M_-}:M_-\to \sphere_C$ is a submersion.
\item\label{2} If $l=m$, the image of $\pi_C$ is $\sphere_C$.
\item\label{3} If $l \geq m+1$, the map $\pi_C$ is surjective onto $\DD_C$ and its restriction to the regular part is a submersion.
\item\label{4} If $l> m+1$, the fibers of $\pi_C$ are connected.
\item\label{5} If $l=m+1$, $C$ can be extended to a Clfford system $C'$ of rank $m+2$, the image of $\pi_{C'}$ is $\sphere_{C'}=\sphere^{m+1}$ and $\pi_C$ factors as $\pi_C=Pr\circ \pi_{C'}$, where $Pr:\sphere_{C'}\to \DD_C$ is given by
\[
Pr(x_1,\ldots x_m, x_{m+1})=(x_1,\ldots x_m).
\]
In particular, the fibers of $\pi_C$ are not connected.
\end{enumerate}
\end{prop}
\begin{proof}
\ref{1}) Let $x_0\in \sphere^{2l-1}$ and $P=\pi_C(x_0)$. $P$ lies in $\sphere_C$ is and only if $\|P\|=1$. The inequality \eqref{inequality} is then an equality, and in particular $\scal{Px_0,x_0}=\|Px_0\|\cdot \|x_0\|$, which implies that $x_0$ is an eigenvector for $P$. Since $P$ has eigenvalues $\pm 1$, $Px_0=\pm x_0$, and again from $\scal{Px_0,x_0}=1$ is must be $Px_0=x_0$.

On the other hand, if $x_0\in E^1_+(P )$ for some $x_0\in \sphere^{2l-1}$, by \eqref{property}
\[
\scal{P,\pi_C(x_0)}=\scal{Px_0,x_0}=1
\]
and therefore $P=\pi_C(x_0)$. Thus the whole unit sphere $E^1_+(P )$ projects to $P$. In particular, $M_-$ embeds in $\sphere^{2l-1}\times \sphere_C$ as $M_-=\{(x,P)\in \sphere^{2l-1}\times \sphere_C| Px=x\}$ and $\pi_C$ is just the projection onto the second factor, which can easily be checked to be a submersion.

\ref{2}) Fix an orthonormal basis $(P_0,\ldots P_m)$ of $\RR_C$. Given $x\in \sphere^{2m-1}$, let $x=ax_++bx_-$ where $x_{\pm}\in E_{\pm}(P_0)$ are unit vectors, and $a^2+b^2=1$. We want to prove that
\[
\sum_{i=0}^m \scal{P_ix,x}^2=1.
\]
For $i=0$ we have $\scal{P_0x,x}=a^2-b^2$, while for $i=1,\ldots m$ we compute
\[
\scal{P_ix,x}=a^2\scal{P_ix_+,x_+}+b^2\scal{P_ix_-,x_-}+2ab\scal{P_ix_+,x_-}
\]
On the one hand, since $P_ix_\pm\in E_\mp(P_0)$ for $i=1,\ldots m$, the equation above simplifies as $\scal{P_ix,x}=2ab\scal{P_ix_+,x_-}$. On the other hand, since $m=l=\dim E_-(P_0)$, the vectors $P_1x_+, \dots P_mx_+$ form an orthonormal basis of $E_-(P_0)$ and thus
\[
\sum_{i=1}^{m}\scal{P_ix,x}^2=4a^2b^2\sum_{i=1}^m\scal{P_ix_+,x_-}^2=4a^2b^2.
\]
Therefore
\[
\sum_{i=0}^m \scal{P_ix,x}^2=\scal{P_0x,x}^2+\sum_{i=1}^m \scal{P_ix,x}^2=(a^2-b^2)^2+4a^2b^2=1
\]
as we wanted. Moreover, since the preimage of any $P\in \sphere_C$ consists of the unit sphere in $E_+(P)$, it is non empty and thus the image of $\pi_C$ is $\sphere_C$.

\ref{3}) Fix an orthonormal basis $(P_0,\ldots P_m)$ of $\RR_C$ and let $x_+\in E_+^1(P_0)$. As $P_0$ anticommutes with $P_i$, $i>1$, we have $P_ix_+\in E_-(P_0)$. If $l\geq m+1$ there is a unit vector $x_-\in E_-(P_0)$ which is perpendicular to $P_1x_+, \ldots P_mx_+$, and let $x={\sqrt{2}\over 2}(x_++x_-)\in \sphere^{2l-1}$. It is easy to check that $\pi_C(x)=0$, and therefore the preimage of the origin (which is the manifold $M_+$ as observed in Remark \ref{R:M+}) is nonempty in this case. Moreover, the set
\[
M_{(Q,t)}=\{\cos(t)x+\sin(t)Qx,\; x\in M_+\}, \qquad Q\in \sphere_C,\; t\in [0,\pi/4]
\]
is contained in (and by dimensional reasons it coincides with) the preimage of the point $\sin(2t)Q$. Since any point in $\DD_C$ can be written in this way, it follows that $\pi_C$ is surjective onto $\DD_C$.
Moreover for any $P\in \sphere_C$ the gradient of $x\mapsto \scal{Px,x}$ in $\sphere^{2l-1}$ is $X_P(x)=2Px-2\scal{Px,x}x$. If $x$ projects to the interior of $\DD_C$, the set $\big\{X_{P_0}(x), \ldots X_{P_m}(x)\big\}$ is linearly independent and it spans a $m+1$-dimensional subspace of $T_x\sphere^{2l-1}$ orthogonal to the fibers of $\pi_C$, thus projecting onto $T_{\pi_{_C}(x)}\DD_C$. Therefore $\pi_C$ is a submersion.

\ref{4}) Fix an orthonormal basis $(P_0,\ldots P_m)$ of $\RR_C$ and take $x_+\in E^1_+(P_0)$. On $E_-(P_0)$, consider the orthogonal complement $V_{x_+}$ of $span(P_1x,\ldots P_mx)$, and take its unit sphere $V_{x_+}^1\in \sphere^{2l-1}$. The dimension of $V_{x_+}^1$ is $l-m-1$ and for every $x_-\in V_{x_+}^1$ the element $x={\sqrt{2}\over 2}(x_++x_-)\in \sphere^{2l-1}$ satisfies $\pi_C(x)=0$ and thus $x$ belongs to $M_+$.

Taking the union of all $V_x^1$ as $x$ varies in $E_+^1(P_0)$, we obtain a sphere bundle $V^1\ra E_+^1(P_0)$ whose fiber has dimension $l-m-1$. In particular, if $l>m+1$ the fiber is connected, and so is $V^1$. As we have a surjective map $V^1\to M_+$ sending $y\in V_x^1$ to ${\sqrt{2}\over 2}(x+y)$, $M_+$ is connected as well. Finally, since all the fibers of points in the interior of $\DD_C$ are homeomorphic to each other (and, in particular, to $M_+$), every fiber is connected.

\ref{5}) If $l=m+1$, by table \ref{E:table-delta-m} it follows that $m=1,3,7$ and for all cases $m$ is not a multimple of $4$. Given a Clifford system $C'$ of rank $m+2$ in $\RR^{2l}$, by the uniqueness of Clifford systems for $m\not\equiv 0 (\mmod 4)$ it follows that $C$ is algebraically equivalent to a sub-Clifford system of $C'$. We can thus find an orthonormal basis $(P_0,\ldots P_{m+1})$ of $\RR_{C'}$ such that $(P_0,\ldots P_m)$ is a basis for $\RR_C$. Since we can express $\pi_C(x)$ as $(\scal{P_0x,x},\ldots \scal{P_{m}x,x})$ and similarly for $\pi_{C'}$, $\pi_C$ factors as $\pi_C=Pr\circ \pi_{C'}$, where $Pr:\sphere_C\to \DD_C$ is given by $(x_0,\ldots, x_{m},x_{m+1})\mapsto (x_0,\ldots, x_{m})$.
\end{proof}
\begin{rem}
Since we are interested in having connected fibers, we will not consider from now on the Clifford systems with $l=m+1$.
\end{rem}

\begin{prop}\label{FC-transnormal}
Let $C$ be a Clifford system of rank $m+1$ on $\RR^{2l}$. The fibers of $\pi_C$ define a transnormal system on $\sphere^{2l-1}$, whose leaf space is $\DD_C$ (if $l>m+1$) or $\sphere_C$ (if $l=m$) with a round metric of curvature $4$.
\end{prop}
\begin{proof}
We  prove the proposition when the quotient is $\DD_C$, the other case follows in a similar fashion. In order to prove the proposition, we consider the family $\mathfrak{F}$ of geodesics in $\sphere^{2l-1}$ given by
\[
\mathfrak{F}=\{\gamma(t)=\cos(t) x_-+\sin(t) x_+|\;P\in \sphere_C,\;x_\pm\in E^1_\pm(P)\}
\]
and we show that the following properties hold:
\begin{enumerate}
\item Every geodesic in $\mathfrak{F}$ is orthogonal to the fibers of $\pi_C$ at all points.
\item For every point $x\in \sphere^{2l-1}$ and vector $z$ normal to the fiber of $\pi_C$ through $x$, there is a geodesic in $\mathfrak{F}$ passing through $x$ and tangent to $z$.
\item Every geodesic in $\mathfrak{F}$ projects to a unit speed geodesic in $\DD_C$.
\end{enumerate}

1) If $x\in \sphere^{2l-1}$ projects to a point in the interior of $\DD_C$, the normal space of the fiber through $x$ is spanned by the vectors $X_{P_i}(x)=P_ix-\scal{P_ix,x}x$. On the other hand, if $x$ projects to $P\in\partial \DD_C$ then the fiber through $x$ is $E^1_+(P)$ and its normal space is just $E_-(P)$. Any geodesic $\gamma\in \mathfrak{F}$, $\gamma(t)=\cos(t) x_-+\sin(t) x_+$ for some $x_{\pm}\in E_\pm(P)$ is by definition perpendicular in $x_+$, $x_-$ to the corresponding fibers. For $t\in (0,\pi/2)$ we have $P\gamma(t)=-\cos(t)x_-+\sin(t)x_+$, $\scal{P\gamma(t),\gamma(t)}=-\cos(2t)$ and it is just matter of computations to show that $\gamma'(t)={1\over \sin(2t)} \cdot X_{P}(\gamma(t))$:
\begin{align}\label{equation-connecting-geods}
 X_P(\gamma(t))=&P\gamma(t)-\scal{P\gamma(t), \gamma(t)}\gamma(t)\\
 =&(-\cos(t) x_-+\sin(t) x_+)+\cos(2t)(\cos(t) x_-+\sin(t) x_+) \nonumber \\
=&-\cos(t) x_-+\sin(t) x_++\cos(2t)(\cos(t) x_-+\sin(t) x_+)\nonumber \\
=&\sin(2t)(-\sin(t)x_-+\cos(t)x_+) \nonumber\\
=&\sin(2t)\gamma'(t)\nonumber
\end{align}

2) If $x\in \sphere^{2l-1}$ projects to $P\in \partial \DD_C$, then it belongs to the positive eigenspace $E_+^1(P)$ and, if $z$ is perpendicular to the fiber through $x$, then it belongs to $E_-(P)$. Therefore, $\gamma(t)=\cos(t)x+\sin(t)z$ belongs to $\mathfrak{F}$ and it satisfies $\gamma(0)=x,\, \gamma'(0)=z$. If $x$ projects to a point in the interior of $\DD_C$, any vector $z$ normal to the fiber through $x$ is of the form $z=X_P(x)$ for some $P\in \sphere_C$. Such a $P$ gives a splitting $\RR^{2l}=E_+(P)\oplus E_-(P)$, and $x$ can be written as $x=\cos(t_0)x_-+\sin(t_0) x_+$ for some $x_{\pm}\in E_{\pm}(P)$. Equation \eqref{equation-connecting-geods} says that $z$ is parallel to $\gamma'(t_0)$, where $\gamma(t)=\cos(t)x_-+\sin(t)x_+$ is in $\mathfrak{F}$.

3) Notice first that the unit speed geodesics in $\DD_C$ with the round metric of constant curvature $4$ are of the form $\cos(2t)P+\sin(2t) Q$ where $P,Q\in \DD_C$ satisfy $\scal{P,Q}=0$. Given $\gamma(x)= \cos(t)x_-+\sin(t)x_+$ for some $x_{\pm}\in E_\pm^1(P)$, $P\in \sphere_C$, let $Q=\sum \scal{P_ix_+,x_-}P_i$.
Again it is just a computation to check that $\|Q\|^2\leq 1$ and thus $Q\in \DD_C$, $\scal{-P,Q}=0$, and $\pi_C(\gamma(t))=-\cos(2t)P+\sin(2t)Q$. Therefore $\pi_C(\gamma(t))$ is a geodesic in $\DD_C$, as we wanted to show.
\end{proof}
\begin{defn}
Given a Clifford system $C$ of rank $m+1$ in $\RR^{2l}$, with $l\neq m+1$, we define the \emph{Clifford foliation} $\fol_C$ to be the foliation on $\sphere^{2l-1}$ given by the fibers of $\pi_C$.
\end{defn}

\begin{rem}
By Proposition \ref{FC-transnormal}, any Clifford foliation $\fol_C$ is a transnormal system. In fact, we will prove that $\fol_C$ is a singular Riemannian foliation. This requires proving the existence of smooth vector fields spanning the leaves of the foliation. We will prove this in Proposition \ref{P:SRF} for a larger class of foliations that includes the Clifford foliations.
\end{rem}
\begin{cor}
If $C$ is a Clifford system of rank $m+1$ on $\RR^{2l}$ and $l=m$, $\pi_C:\sphere^{2m-1}\to \sphere_C$ is a Hopf fibration.
\end{cor}
\begin{proof}
In this particular case $\pi_C$ is a submersion, and by Proposition \ref{FC-transnormal} it is Riemannian. By the work of Grove and Gromoll \cite{GG} and Wilking \cite{Wil}, the submersion $\pi_C$ must be in fact a Hopf fibration.
\end{proof}

\subsection{Symmetries of the Clifford foliations}\label{SS:Symmetries-Clifford-Foliations}
We discuss here some natural symmetries of the Clifford foliations $(\sphere^{2l-1},\fol_C)$.
Let $P$ be an element of $\sphere_C$. For any $x, y\in \sphere^{2l-1}$, we have $\scal{Px,Py}=\scal{P^2x,y}=\scal{x,y}$ and therefore the elements of $\sphere_C$ are also orthogonal maps on $\RR^{2l}$. Moreover, by definition of $\pi_C$ we have
\[
\scal{\pi_C(Px),Q}=\scal{QPx,Px}\qquad \forall Q\in \sphere_C.
\]
Since $QP=-PQ+2\scal{P,Q} Id$, the equation before becomes
\begin{align*}
\scal{\pi_C(Px),Q}&=-\scal{PQx,Px}+2\scal{P,Q}\scal{x,Px}\\
&= -\scal{\pi_C(x),Q}+2\scal{P,Q}\scal{\pi_C(x),P}\\
&= \scal{-\pi_C(x)+2\scal{\pi_C(x),P}P, Q}
\end{align*}
Therefore, $\pi_C(Px)=-\pi_C(x)+2\scal{\pi_C(x),P}P$ and therefore there is a commutative diagram
\[\begindc{\commdiag}[50]
\obj(0,0)[Q1]{$\DD_C$} \obj(2,0)[Q2]{$\DD_C$} \obj(0,1)[S1]{$\sphere^{2l-1}$} \obj(2,1)[S2]{$\sphere^{2l-1}$}
\mor{S1}{Q1}{$\pi_C$} \mor{S2}{Q2}{$\pi_C$} \mor{S1}{S2}{$P$} \mor{Q1}{Q2}{$\rho_P$}
\enddc\]
where $\rho_P$ is the reflection of $\DD_C$ along the segment through $P$. The subgroup of $O(2l)$ generated by the elements $P\in \sphere_C$ is usually denoted $Pin(m+1)$. Its subgroup generated by the products $PQ$, $P,Q\in \sphere_C$, is $\Spin(m+1)$. Since any element $P\in \sphere_C$ can be thought as a foliated isometry of $(\sphere^{2l-1},\fol_C)$, there is a map $\eta:\Spin(m+1)\to \SO(2l)$ whose induced action on $\DD_C$ is isometric and has cohomogeneity 1. The origin $\pi_C(M_+)$ is the only singular orbit of this action, the boundary $\pi_C(M_-)$ consists of one orbit, and the quotient $\DD_C/\Spin(m+1)$ is isometric to $[0,\pi/4]=\sphere^{2l-1}/\fol_C'$, where $\fol_C'$ is the FKM example corresponding to the Clifford system $C$.

%
%

\section{Composed foliations}\label{s:composed}

The goal of this section is to employ this method, introduced by A. Lytchak, to produce new singular Riemannian foliations on spheres out of the Clifford foliations.

Fix a Clifford system $C$ of rank $m+1$ on $\RR^{2l}$. From Proposition \ref{FC-transnormal} we know that the leaf space of a Clifford foliation is isometric to either ${1\over 2}\sphere_C$ (i.e., $\sphere_C$ with a metric of constant curvature $4$) or $\DD_C$ with a hemisphere metric of curvature $4$. With such a metric, $\DD_C$ can be described metrically as a spherical join ${1\over 2}(\sphere_C\star \{pt\})$, where the factor ${1\over 2}$ in front denotes a rescaling of the metric by a factor $1\over 2$.

Let $(\sphere_C,\fol_0)$ be a closed transnormal system on $\sphere_C$, with leaf space $\Delta$ and projection $\pi_0:\sphere_C\to \Delta$. If $l=m$ the composition $\pi_0\circ \pi_C$ gives a submetry $\sphere^{2l-1}\to {1\over 2}\Delta$. If $l>m+1$, the submetry $\pi_0:\sphere_C\to \Delta$ induces a submetry
\[
\hat{\pi}_0:\textstyle\frac{1}{2}(\sphere_C\star\{pt\})\to {1\over 2}(\Delta\star\{pt\}).
\]
Composing $\hat{\pi}_0$ with $\pi_C:\sphere^{2l-1}\to {1\over 2}(\sphere_C\star\{pt\})$, we again obtain a submetry $\hat{\pi}_{0}\circ\pi_C:\sphere^{2l-1}\to {1\over 2}(\Delta\star\{pt\})$.

In either case, we obtain a submetry $\sphere^{2l-1}\to \ul{\Delta}$, where $\ul{\Delta}={1\over 2}\Delta$ or ${1\over 2}(\Delta\star\{pt\})$, and the fibers of this submetry are by construction the leaves of a transnormal system on $\sphere^{2l-1}$, which we denote by $\fol_0\circ \fol_C$.

The goal of this section is to prove the following result.

\begin{prop}\label{P:SRF}
If $(\sphere^{2l-1},\fol_C)$ is a Clifford foliation and $(\sphere_C,\fol_0)$ is a singular Riemannian foliation, then $\fol_0\circ \fol_C$ is a singular Riemannian foliation as well.
\end{prop}

Once again we prove the result in the case where the quotient is $\DD_C$, the other case being essentially contained in this one. 

The foliation $\fol_0\circ \fol_C$ can be described in the following equivalent way. The singular Riemannian foliation $(\sphere_C,\fol_0)$ can be extended to a singular Riemannian foliation $\fol_0^h$ on $\DD_C$, by defining the leaf $L_{tP}$ through $tP$ as $t\cdot L_P$, where $P\in \sphere_C$ and $t\in [0,1]$. The foliation $\fol_0\circ\fol_C$ is then given by the preimages under $\pi_C:\sphere^{2l-1}\to \DD_C$ of the leaves in $\fol_0^h$.

Let $\mathring \DD_C$ denote the interior of $\DD_C$. Since $(\mathring \DD_C,\fol_0^h)$ is a singular Riemannian foliation and $\pi_C:\sphere^{2l-1}\setminus M_-\to \mathring \DD_C$ is a Riemannian submersion, in particular $\fol_0\circ\fol_C$ is a singular Riemannian foliation on $\sphere^{2l-1}\setminus M_-$. Similarly, since $(\sphere_C,\fol_0)$ is a singular Riemannian foliation and $\pi_C|_{M_-}:M_-\to \sphere_C$ is a Riemannian submersion, the restriction $(M_-,(\fol_0\circ\fol_C)\big|_{M_-})$ is again a singular Riemannian foliation.

What we are left to prove, is that for every point $x\in M_-$ there exists a neighbourhood of $x$ in $\sphere^{2l-1}$ in which the restriction of $\fol=\fol_0\circ\fol_C$ is a singular foliation. 
\begin{prop}
Consider a neighborhood $U\In M_-$ of a point $x\in M_-$, small enough that $\nu(M_-)|_U$ admits an orthonormal frame $\{\xi_1,\ldots \xi_{r}\}$, $r=\codim(M_-)$. Then the trivialization
\begin{align*}
\rho: U\times \DD^{r}(\epsilon)&\lra Tub_{\epsilon}(U)\\
(x,(a_1,\ldots a_r))&\lmt \exp_x\left(\sum a_i\xi_i(x)\right)
\end{align*}
is a diffeomorphism, and $\rho^*(\fol|_{Tub_\epsilon(U)})= \fol|_{U}\times \fol_{\DD^r}$, where $(\DD^r(\epsilon),\fol_{\DD^r})$ is the foliation by concentric spheres around the origin. In particular, $\fol|_{Tub_\epsilon(U)}$ is a singular foliation around $M_-$.
\end{prop}

\begin{proof}
%

It is clear that $\rho$ is a diffeomorphism. We will now show that $\rho$ induces a bijection among the leaf spaces.

Consider $U\times [0,\epsilon]$, together with the foliation $\fol|_U\times \{pts\}$. The map
\begin{align*}
(id, r): (U\times \DD^r(\epsilon),\fol|_U\times \fol_{\DD^r})& \lra (U\times [0,\epsilon],\fol|_U\times\{pts\})\\
(u,v)&\lra (u,\|v\|)
\end{align*}
clearly induces a bijection among leaf spaces. Moreover, let $\mathsf{p}:\Tub_{\epsilon}(U)\to U$ denote the metric projection, and consider the map
\begin{align*}
(\mathsf{p}, d_U): (\Tub_{\epsilon}(U),\fol)& \lra (U\times [0,\epsilon],\fol|_U\times\{pts\})\\
x&\lra (\mathsf{p}(x), dist(x,U))
\end{align*}
This map takes the leaf $M_{([P],t)}=\{\cos(t)x+\sin(t)Qx|\; x\in M_+,\, Q\in L_P\}$ of $\fol$ to the leaf $L_{P}\times\{\pi/4 - t\}$. Since every leaf of $\fol$ in $\Tub_{\epsilon}(U)$ is uniquely determined by $[P]\in U/\fol|_U$ and $t\in [\pi/4-\epsilon, \pi/4]$, it follows that $(\mathsf{p},d_U)$ induces a bijection between the leaf spaces as well.

Finally, we claim that $(\mathsf{p}, d_U)\circ \rho=(id, r)$.
\begin{align*}
(\mathsf{p},d_U)\Big(\rho\big(u, (a_1,\ldots a_r)\big)\Big)&= (\mathsf{p},d_U)\left(\exp_{u}\sum a_i\xi_i(u)\right)\\
&= \left(\mathsf{p}\left(\exp_{u}\sum a_i\xi_i(u)\right), dist\left(\exp_{u}\sum a_i\xi_i(u), U\right)\right)\\
&= \big(u, \|(a_1,\ldots a_r)\|\big)
\end{align*}
In particular, $\rho$ induces a bijection between the leaf spaces as well, and this finishes the proof.
\end{proof}

\begin{rem}
If $\fol_0$ is a trivial foliation whose leaves consist of points, $\fol_0\circ \fol_C=\fol_C$ and in particular, $\fol_C$ is a singular Riemannian foliation. Moreover, when $C$ is a Clifford system of rank $m+1$ on $\RR^{2l}$ and $l=m+1$, the foliation $\fol_C$ given by the (non connected) fibers of $\pi_C$ is a \emph{singular Riemannian foliation with disconnected fibers} as defined in \cite[Sect. 3]{AR}.
\end{rem}

%
%

\section{Rigidity of Clifford foliations}\label{s:rigidity}
In the FKM families, the map that associates an isoparametric foliation to each (geometric equivalence class of) Clifford system is neither injective, nor surjective. In fact, on the one hand there are examples of geometrically distinct Clifford systems giving rise to the same isoparametric foliation. On the other hand, there are isoparametric foliations that do not come from a Clifford algebra, whose quotient is isometric to that of the FKM examples.

The goal of this section is to prove that in our case, the map $C\mapsto \fol_C$ described in the previous sections does determine a bijection between the geometric equivalence classes of Clifford system, and the congruence classes of singular Riemannian foliations in spheres whose quotient is a sphere or hemisphere of curvature $4$. We will prove this in the next two propositions.
\begin{prop}\label{P:rigidity}
Suppose $(\sphere^n, \fol)$ is a singular Riemannian foliation such that the quotient space is a hemisphere ${1\over 2}\DD^{m+1}$ of constant curvature $4$. Then $\fol=\fol_C$ for some Clifford system $C$.
\end{prop}
\begin{proof}
Consider the boundary ${1\over 2}\sphere^{m}$ of ${1\over 2}\DD^{m+1}$. Take an \emph{orthonormal basis} of ${1\over 2}\sphere^m$, i.e. $m+1$ points $p_0,\ldots p_m\in {1\over 2}\sphere^m$ mutually at distance $\pi/4$. Given a point $p_i$, the partition of ${1\over 2}\DD^{m+1}$ into the distance spheres around $p_i$ and $-p_i$ lifts via $\pi:\sphere^n\to {1\over 2}\DD^{m+1}$ to a codimension $1$ foliation $\fol^*$ of $\sphere^n$ whose quotient is an interval of length $\pi/2$. By Cartan's classification of such foliations, it follows that the singular leaves of $\fol^*$, i.e. the leaves of $\fol$ corresponding to ${\pm p_i}$, are totally geodesic subspheres of $\sphere^n$, and since they lie on the same stratum they must have the same dimension, call it $l$. In particular, $n=2l-1$ and $\RR^{n+1}=\RR^{2l}$ splits orthogonally as $V_+(p_i)\oplus V_-(p_i)$, where $V_{\pm}(p_i)$ is the space containing the great sphere $\pi^{-1}(\pm p_i)$. Define a linear map $P_i\in \Sym^2(\RR^{2l})$ by
\[
P_i|_{V_+(p_i)}=id,\qquad P_i|_{V_-(p_i)}=-id.
\]
Notice that by definition $P_i^2=id$ and $E_{\pm}(P_i)=V_{\pm}(p_i)$. This produces maps $(P_0,\ldots, P_m)\in \Sym^2(\RR^{2l})$. In order to conclude the proof, it will be enough to prove that $P_iP_j=-P_jP_i$ for $i\neq j$, or equivalently, that $P_i(E_{\pm}(P_j))=E_{\mp}(P_j)$.

It is enough to show that $P_0(E_{+}(P_1))=E_-(P_1)$. Take a point $x\in E_+(P_0)$ in the preimage of $p_0$, and take a horizontal geodesic $\gamma$ starting at $x$ and tangent to the singular stratum, such that $\pi(\gamma)$ passes through $p_1$. Since $\pi(\gamma)(\pi/2)=-p_0$, the point $y=\gamma(\pi/2)$ belongs to $E_-(P_0)$ and we can write $\gamma(t)=\cos(t)x+\sin(t)y$. Moreover, $w=\gamma(\pi/4)={\sqrt{2}\over 2}x+{\sqrt{2}\over 2}y$ belongs to $E_+(P_1)$ by construction of $\gamma$. Then
\[
P_0(w)=P_0\left({\sqrt{2}\over 2}x+{\sqrt{2}\over 2}y\right)={\sqrt{2}\over 2}x-{\sqrt{2}\over 2}y=\gamma(-\pi/4)
\]
But $\pi(\gamma)(-\pi/4)=-p_1$, that means $P_0(w)\in E_-(P_1)$.

Since any $w\in E_+^1(P_0)$ can be written as $\gamma(\pi/4)$ for some horizontal geodesic $\gamma$ from $E^1_+(P_0)$ and $E_-^1(P_0)$, we obtain that $P_0(E_+(P_1))\In E_-(P_1)$. Since $P_0$ is nonsingular, by dimensional reasons it must be $P_0(E_+(P_0))=E_-(P_0)$ and this finishes the proof.
\end{proof}

\begin{prop}\label{P:distinguish}
The Clifford foliations $(\sphere^{2l-1},\fol_C)$ distinguish the geometric equivalence classes of Clifford systems. In other words, if $C$ and $C'$ are geometrically inequivalent Clifford systems on $\RR^{2l}$ and $\RR^{2l'}$ respectively, then there are no foliated isometries between $(\sphere^{2l-1},\fol_C)$ and $(\sphere^{2l'-1},\fol_{C'})$.
\end{prop}
\begin{proof}
Let $(P_0,\ldots P_m)$ be an orthonormal basis of $\RR_C$ and $(Q_0,\ldots Q_{m'})$ an orthonormal basis of $\RR_{C'}$. Since the leaf spaces of $\fol_C$, $\fol_{C'}$ have dimension $m+1$, $m'+1$ respectively, it follows immediately that $\fol_C\neq \fol_{C'}$ unless $m=m'$. If $m=m'$, we have $l=k\delta(m)$, $l'=k'\delta(m')=k'\delta(m)$ and therefore $\fol_C\neq \fol_{C'}$ unless $k=k'$ as well.

Assume now that $m=m'$ and $k=k'$. As we recalled in Section \ref{S:Clifford-syst}, if $m\not\equiv 0 (\mmod 4)$ there is only one geometric class of Clifford systems for each $k$, and therefore $\fol_C=\fol_{C'}$. If $m\equiv0 (\mmod 4)$ then the geometric class of $C$ is uniquely determined by the non-negative integer $|\tr(P_0\cdot\ldots \cdot P_m)|$. Therefore the last thing remained to prove is that $\fol_C$ and $\fol_{C'}$ are not congruent unless $|\tr(P_0\cdot\ldots \cdot P_m)|=|\tr(Q_0\cdot\ldots \cdot Q_m)|$. This was already established in \cite[page 486]{FKM}, as they showed that the invariant $|\tr(P_0\cdot\ldots \cdot P_m)|$ represents a characteristic number of the vector bundle $E\to \sphere_C$ whose sphere bundle is $\pi_C|_{M_-}:M_-\to \sphere_C$. 

\end{proof}

%
%

\section{Homogeneous foliations}\label{s:homogeneous}
In this section we investigate the Clifford and composed foliations that are homogeneous.
\subsection{Clifford foliations}
When $\sphere^{2l-1}/\fol_C$ is a sphere ${1\over 2}\sphere^m$, it is known that the foliation is homogeneous if and only if $m=2$ or $4$. Therefore we can restrict our attention to the case where the quotient is a hemisphere. Our first result restricts the list of possible homogeneous Clifford foliations.
\begin{prop}\label{homogeneous}
Let $C$ be a Clifford system of rank $m+1$ on $\RR^{2l}$ such that $l>m+1$. On $\sphere^{2l-1}$ consider the Clifford foliation $\fol_C$ and the FKM isoparametric family $\fol_C'$ associated to $C$. If $\fol_C$ is homogeneous, then $\fol_C'$ is homogeneous as well.
\end{prop}
\begin{proof} 
Suppose that $(\sphere^{2l-1},\fol_C)$ is given by an isometric action of some Lie group $H\In \SO(2l)$. Let $G\In \SO(2l)$ be the closure of the group generated by $H$ and the image of the spin representation $\eta:\Spin(m+1)\to \SO(2l)$ defined in Section \ref{SS:Symmetries-Clifford-Foliations}. Since both $H$ and $\Spin(m+1)$ act by foliated isometries on $(\sphere^{2l-1}, \fol_C)$, so does  $G$. Moreover, the $G$ action descends to a cohomogeneity $1$ action on $\DD_C$. In particular, the $G$-orbits in $\sphere^{2l-1}$ correspond to the leaves of $\fol_C'$, and therefore $\fol'_C$ is homogeneous.
\end{proof}

From Proposition \ref{homogeneous} above and the table in section \ref{constructionFKM} it follows that the only possible homogeneous Clifford foliations with $l>m+1$ come from Clifford systems with $(m,k)=(1,k), (2,k), (9,1)$,  or $m=4$ and $P_1\cdot\ldots \cdot P_4=\pm Id$.
\begin{prop}\label{P:Clifford-homog}
Let $C$ be a Clifford system of rank $m+1$ on $\RR^{2l}$, $l=k\delta(m)$. Then:
\begin{itemize}
\item If $m=1$, $\fol_C$ is given by the orbits of the diagonal $\SO(k)$-action on $\sphere^{2k-1}\In \RR^k\oplus \RR^k$.
\item If $m=2$, $\fol_C$ is given by the orbits of the diagonal $\SU(k)$-action on $\sphere^{4k-1}\In \CC^k\oplus \CC^k$.
\item If $m=4$ and $P_0\cdot P_1\cdot P_2\cdot P_3\cdot P_4=\pm Id$, $\fol_C$ is given by the orbits of the diagonal $\Sp(k)$-action on $\sphere^{8k-1}\In \HH^k\oplus \HH^k$.
\end{itemize}
\end{prop}
\begin{proof} This proof is essentially a version of \cite[Theorem 6.1]{FKM}, adapted to our situation. The Clifford systems with $m=1,2$ or $m=4$ and $P_0\cdot P_1\cdot P_2\cdot P_3\cdot P_4=\pm Id$ can be obtained in the following way: let $\FF\in\{\RR,\CC,\HH\}$ be the division algebra such that $\dim_{\RR}\FF=m$, and let $j_1,\ldots j_{m-1}$ the canonical imaginary units of $\FF$. For $q=q_0+q_1j_1+\ldots q_{m-1}j_{m-1}\in \FF$, $q_i\in \RR$, we define the \emph{real part} of $q$ by $\Re(q)=q_0$ and the \emph{$r$-th imaginary part} of $q$ by $\Im_r(q)=q_r=\Re(q\cdot j_r)$, $r=1,\ldots m-1$.

On $\RR^{2\delta(m)}=\FF^k\times \FF^k$, let $C=(P_0,\ldots P_m)$ be the Clifford system given by
\begin{align*}
&P_0(u,v)=(u,-v),\quad P_1(u,v)=(v,u),\quad P_{r+1}(u,v)=(-j_r\cdot v, j_r\cdot u)\\
&r\in \{1,m-1\}
\end{align*}
where $u,v\in \FF^k$, $u=(u_1,\ldots u_k)$, $v=(v_1,\ldots v_k)$.
The projection $\pi_C$ is determined by the functions
\[
\left\{\begin{array}{rcl}
\scal{P_0(u,v), (u,v)}&=&\|u\|^2-\|v\|^2\\
\scal{P_1(u,v), (u,v)}&=&2\Re(\sum_iu_i\cdot \bar{v}_i)\\
\scal{P_{r+1}(u,v), (u,v)}&=&2\Re(\sum_iu_i\cdot \bar{v}_i\cdot j_r)=2\Im_r(\sum_iu_i\cdot \bar{v}_i)
\end{array}
\right.
\]
and thus we can write
\[
\pi_C(u,v)=(\|u\|^2-\|v\|^2, 2\sum_iu_i\cdot \bar{v}_i)\in \RR\oplus\FF.
\]
The group $\U(\FF,k)$ defined by $\U(\FF,k)=\SO(k), \SU(k)$ or $\Sp(k)$ according to whether $\FF=\RR, \CC$ or $\HH$ respectively, acts transitively on $\sphere^{mk-1}\In \FF^k$ and its diagonal action on $\FF^k\times \FF^k$ preserves the functions $f_i$, $i=0,\ldots m$. In particular, the orbits of such action are contained in the fibers of $\pi_C$, and therefore in the leaves of $\fol_C$. Moreover, any point $(u,v)\in \sphere^{2mk-1}\In \FF^k\times \FF^k$ can be moved by the $\U(\FF,k)$-action to a point of the form
\[
(u_1e_1, v_1 e_1+v_2e_2)
\]
where $e_1, e_2$ are elements of the canonical basis on $\FF^k$, $v_1\in \FF$, $u_1,v_2\in \RR_{\geq 0}$ and $u_1^2+|v_1|^2+v_2^2=1$. It is easy to see that such $u_1, v_1, v_2$ are uniquely determined by the functions $f_i$, and therefore there is only one such point for each fiber of $\pi_C$. In particular,  every point in a fiber of $\pi_C$ can be moved to a specific point via the action of $\U(\FF,k)$, and therefore the orbits of $\U(\FF,k)$ coincide with the leaves of $\fol_C$.
\end{proof}

On the other hand, the remaining foliation $\fol_C$, $C=C_{9,1}$ on $\RR^{32}$, is not homogeneous.
\begin{prop}
The Clifford foliation induced by the Clifford system $C=C_{9,1}$ on $\RR^{32}$ is not homogeneous.
\end{prop}
\begin{proof}
Suppose $(\sphere^{31}, \fol_C)$ is homogeneous, given by the orbits of a group $G\In \SO(32)$.

First of all, we prove that the principal isotropy group $H$ must be trivial. If not, consider the subsphere $\sphere^h=\Fix(H)$, and take $G'=N(H)/H$, where $N(H)$ is the normalizer of $H$ in $G$. The identity component $G'_0$ of $G'$ acts on $\sphere^h$ with trivial principal groups and there is an orbifold cover $\sphere^h/G'_0\to \sphere^{31}/G$, where $\sphere^{31}/G$ is isometric to the hemisphere ${1\over 2}\sphere^{10}_+$. The quotient $\sphere^h/G'_0$ cannot be ${1\over 2}\sphere^{10}$ (the only spheres that can arise as such quotient must have dimension $2,4$ or $8$, see for example the introduction of \cite{LW}) and since $\pi_1^{orb}({1\over 2}\sphere^{10}_+)=\ZZ/2\ZZ$, it must be $\sphere^{h}/G'_0={1\over 2}\sphere^{10}_+$. By Proposition \ref{P:rigidity} it follows that $(\sphere^h,G'_0)$ is itself a Clifford foliation, with respect to some Clifford system $(Q_0,\ldots Q_9)$. In particular $h\geq 31$, and therefore it must be $\sphere^h=\sphere^{31}$ and $H=\{1\}$. 
\\

Since the leaf $E_+^1(P_0)$ is a totally geodesic sphere of dimension $15$, $G$ acts transitively on $\sphere^{15}$ by isometries. On the other hand, since the $G$- action has trivial principal isotropy groups, it must have $\dim G=21$, and this gives a contradiction since there are no groups of dimension 21 that act transitively on $\sphere^{15}$ (see for example \cite[Table C]{GWZ}).
\end{proof}

Finally, we determine the homogeneity of a big fraction of the composed foliations $\fol_0\circ \fol_C$, in terms of the homogeneity of $\fol_0$ and and $\fol_C$.

\begin{prop}\label{P:composed-homogeneous}
Let $C,\, \fol_C,\, \fol_C'$ be defined as in proposition \ref{homogeneous}, and let $(\sphere_C,\fol_0)$ be a singular Riemannian foliation. If the leaf space of $\fol_C$ is a hemisphere and the composed foliation $\fol_0\circ \fol_C$ is homogeneous, then $\fol_0$ and $\fol'_C$ are homogeneous. On the other hand, if $\fol_C$ and $\fol_0$ are homogeneous, so is $\fol_0\circ \fol_C$.
\end{prop}
\begin{proof}
Suppose first that $(\sphere^{2l-1},\fol_0\circ \fol_C)$ is homogeneous, given by the orbits of a $G$-action. Remember that $M_+$ is a leaf for both $\fol_C$ and $\fol_0\circ\fol_C$. For any point $x\in M_+$, the unit normal sphere of $M_+$ at $x$, $\nu_x^1M_+$, is diffeomorphic to $\sphere^m$ via $\pi_C\circ \exp_x^{\perp}$. Moreover (the identity component of) the isotropy group $G_x$ acts on $\nu_x^1M_+$ via the slice representation, whose orbits get mapped to the leaves of $\fol_0$ via the same map $\pi_C\circ \exp_x^{\perp}$ and therefore $\fol_0$ is homogeneous as well. Moreover, as in Proposition \ref{homogeneous} above, we can consider the group $G'\In \SO(2l)$ generated by $G$ and $\eta(\Spin(m+1))$, and the orbits of $G'$ are, once again, the leaves of $\fol_C'$, which is then homogeneous.
\\

Suppose now that $(\sphere^m, \fol_0)$ is homogeneous and it is given by the orbits of a representation $\rho:H\to\SO(m+1)$. Up to a double cover $H'\to H$ we can lift $\rho$ to $\rho':H'\to \Spin(m+1)$, and via the embedding $\eta:\Spin(m+1)\to \SO(2l)$ defined in Section \ref{SS:Symmetries-Clifford-Foliations} we have a representation $\rho'':H'\to \SO(2l)$. By the way we defined $\eta$ it is clear that the $\rho''(H')$-orbits on $\sphere^{2l-1}$ get projected, via $\pi_C$, to $\rho(H)$-orbits on $\DD_C$. In particular, if $\fol_C$ is homogeneous given by some $K$-action, the (closure of the) group $K'\In \SO(2l)$ generated by $K$ and $\rho''(H')$ acts on $\sphere^{2l-1}$ isometrically and the orbits are precisely the leaves of $\fol_0\circ \fol_C$.
\end{proof}
\begin{cor}\label{C:homog}
If  $(\sphere^{2l-1},\fol_C)$ is a Clifford foliation with quotient $\sphere^2$ or $\sphere^4$, then for every singular Riemannian foliation $(\sphere_C,\fol_0)$ the composed foliation $\fol_0\circ\fol_C$ is homogeneous.
\end{cor}
\begin{proof} In this case $\fol_0$ is a singular Riemannian foliation on $\sphere^2$ or $\sphere^4$, and therefore either $\dim\fol_0\leq 3$ or $\fol_0$ is the trivial foliation in $\sphere^4$ consisting of one leaf. In the first case, $\fol_0$ is homogeneous by \cite{Rad}, in the second it is trivially homogeneous. Since $\fol_C$ itself is homogeneous, $\fol_0\circ\fol_C$ is homogeneous by Proposition \ref{P:composed-homogeneous}.
\end{proof}

The results of this sections allow us to prove the second part of Theorem \ref{T:ComposedFoliations}
\begin{prop}\label{P:homog-composed}
Let $C$ be a Clifford system on $\RR^{2l}$ and $(\sphere_C,\fol_0)$ a singular Riemannian foliation. If $C\neq C_{8,1},C_{9,1}$ then $(\sphere^{2l-1},\fol_0\circ\fol_C)$ is homogeneous if and only if $\fol_0$ and $\fol_C$ are homogeneous. If $C=C_{9,1}$ then  $(\sphere^{2l-1},\fol_0\circ\fol_C)$ is homogeneous only if $\fol_0$ is homogeneous.
\end{prop}
\begin{proof}
If $\fol_C$ and $\fol_0$ are homogeneous, then $\fol_0\circ \fol_C$ is homogeneous by \ref{P:composed-homogeneous}. If $\fol_0\circ \fol_C$ is homogeneous then there are two cases to consider:
\begin{itemize}
\item If the leaf space of $\fol_C$ is $\sphere_C$, then $C=C_{2,1}$ or $C_{4,1}$. In both cases $\fol_C$ is homogeneous by Proposition \ref{P:Clifford-homog}, and $\fol_0$ is homogeneous by Corollary \ref{C:homog}.
\item If the leaf space of $\fol_C$ is $\DD_C$, then by Proposition \ref{P:composed-homogeneous} both $\fol_C,\,\fol'_C$ are homogeneous. Moreover, if $C\neq C_{9,1}$ then $\fol_C$ is homogeneous as well by Table \eqref{Table homog} and Proposition \ref{P:Clifford-homog}.
\end{itemize}

\end{proof}

%
%

\section{Clifford foliations on compact rank 1 symmetric spaces}\label{S:CROSS}

The construction of Clifford and composed foliations can be used to produce new foliations on the other simply connected, compact, rank one symmetric spaces.

\subsection{Complex projective spaces}
Let $C$ be a Clifford algebra of rank $m+1$ on $\RR^{2l}$ with $m\geq 1$ and let $(\sphere^{2l-1},\fol_C)$ be the associated Clifford foliation. If we define $i=P_0P_1\in \mathfrak{so}(2l)$, the flow of $i$ defines an isometric $\sphere^1$ action  on $\sphere^{2l-1}$. This action preserves $\fol_C$, and thus it induces an isometric action on the quotient ${1\over 2}\sphere_C$ or $\DD_C={1\over2}(\sphere_C\star\{pt\})$ that acts transitively on the circle containing $P_0, P_1$ while fixing the other elements $P_2,\ldots P_m$. Therefore, the foliation $(\sphere^{2l-1},\fol_C)$ projects to a foliation $\fol_C^{\CC}$ on $\sphere^{2l-1}/\sphere^1=\CC\PP^{l-1}$ with quotient isometric to either ${1\over 2}(\sphere^{m-2}\star \{pt\})$ (a hemisphere of ${1\over 2}\sphere^{m-1}$) or ${1\over 2}(\sphere^{m-2}\star [0,\pi/2])$ (a half hemisphere of ${1\over 2}\sphere^{m}$).
As in the spherical case, given a singular Riemannian foliation $\fol_0$ on $\sphere^{m-2}$ we can define new foliations $(\CC\PP^{l-1},\fol_0\circ\fol_C^{\CC})$.

\subsection{Quaternionic projective spaces}
The case of $\HH\PP^n$ closely follows the construction on $\CC\PP^n$. Let $C$ be a Clifford algebra of rank $m+1$ on $\RR^{2l}$ with $m\geq 1$ and let $(\sphere^{2l-1},\fol_C)$ be the associated Clifford foliation. The Lie algebra generated by $\{i=P_0P_1,\,j=P_1P_2,\,k=P_0P_2\}\In \mathfrak{so}(2l)$ corresponds to a subgroup $\sphere^3\In \SO(2l)$ which acts on $\sphere^{2l-1}$. This action preserves $\fol_C$, and it induces an isometric action on the quotient ${1\over 2}\sphere^m$ or $\DD_C={1\over2}(\sphere^m\star\{pt\})$, that acts transitively on the $2$-sphere containing $P_0,P_1,P_2$ while fixing the other elements $P_3,\ldots P_m$. Therefore, $\fol_C$ projects to a foliation $\fol_C^{\HH}$ on $\sphere^{2l-1}/\sphere^{3}=\HH\PP^{l/2-1}$ whose quotient is isometric to either ${1\over 2}(\sphere^{m-3}\star \{pt\})$ (a hemisphere of ${1\over 2}\sphere^{m-2}$) or ${1\over 2}(\sphere^{m-3}\star [0,\pi/2])$ (a half hemisphere of ${1\over 2}\sphere^{m-1}$). Given a singular Riemannian foliation $\fol_0$ on $\sphere^{m-3}$ we can define new foliations $(\HH\PP^{l/2-1},\fol_0\circ\fol_C^{\HH})$.

\subsection{Cayley projective space}
Let $C'=C_{m,k}$ be a Clifford system on $\RR^{16}$ with $m\in\{3,5,6\}$. Since $m$ does not divide 4, the algebraic equivalence class of $C'$ is uniquely determined by $m$ and therefore $C$ is equivalent to the subsystem $(P_0,\ldots P_m)$ of $C_{8,1}=(P_0,\ldots P_8)$. In particular, the projection $\pi_{C'}:\sphere^{15}\to \DD_{C'}$ factors through ${1\over 2}\sphere_{C_{8,1}}={1\over 2}\sphere^8$ and the leaves of $\fol_{C'}$ contain the fibers of the Hopf fibration $\sphere^{15}\to {1\over 2}\sphere^8$. This means in particular that $(\sphere^{15}, \fol_{C'})$ is obtained by pulling back a foliation on ${1\over 2}\sphere^8$ via the Hopf map and the same can be said about any composed foliation $(\sphere^{15}, \fol_0\circ \fol_{C'})$, for every singular foliation $\fol_0$ on $\sphere_{C'}$. By shrinking the fibers of the Hopf fibration to a factor $t\in (0,1)$, we get a family of metrics $g_t$ on $\sphere^{15}$ for which $(\sphere^{15},g_t)\to {1\over 2}\sphere^8$ is still a Riemannian submersion, and in particular $\fol_{C'}$ is still a singular Riemannian foliation on $(\sphere^{15},g_t)$.

Consider now the Cayley plane $Ca\PP^2$ with its canonical metric. Fixing a point $p_0$, the cut locus of $p_0$ is the sphere of distance $\pi/2$. Moreover, for $r<\pi/2$ the distance sphere of radius $r$ around $p_0$ is isometric to $(\sphere^{15},g_{\cos r})$ while the distance sphere of radius $\pi/2$ is isometric to ${1\over 2}\sphere^8$. Given a singular Riemannian foliation $\fol_0\circ \fol_C$ on $\sphere^{15}=T^1_{p_0}Ca\PP^2$, we induce a foliation $\fol$ on $Ca\PP^2$ by exponentiating the leaves. In other words, we define the leaf through $q=\exp_{p_0}rv$, $\|v\|=1$, as
\[
L_{q}=\{q'=\exp_{p_0}rv'\mid v'\in L_v\}.
\]
Clearly the restriction of $\fol$ to every distance sphere around $p_0$ is a singular Riemannian foliation. Moreover, the set of regular leaves is open and dense, and it is easy to check that the foliation around each regular leaf is defined by the fibers of a Riemannian submersion. In particular, the restriction of $\fol$ to the regular set is a singular Riemannian foliation. Since every singular leaf is a limit of regular leaves, one deduces that singular leaves as well stay at a constant distance from each other, and therefore $\fol$ defines a transnormal system on $Ca\PP^2$. Finally, similarly to Proposition \ref{P:SRF} we can conclude that $\fol$ is, in fact, a singular Riemannian foliation on $Ca\PP^2$. This foliation cannot be homogeneous, because on the unit sphere around $p_0$ it is given by $(\sphere^{15},\fol_0\circ\fol_C)$ which is not homogeneous by Proposition \ref{P:homog-composed}.

\section{Properties of Clifford and composed foliations}\label{s:properties}
The new examples exhibit some behaviours that either do not appear in the homogeneous case, or have not been shown to appear. We collect here a few of these new behaviours.

\subsection{Orbifold quotient}

Recall that a singular Riemannian foliation $(M,\fol)$ is called polar if, for every point $p\in M$, there is a totally geodesic submanifold of dimension equal to the codimension of $\fol$ that passes through $p$ and is perpendicular to all the leaves it meets. The quotient of a closed polar foliation $(\sphere^n,\fol)$ has constant curvature 1. If $(\sphere^{2l-1},\fol_C)$ is a Clifford foliation with hemispherical quotient and $(\sphere_C,\fol_0)$ is a polar foliation, then the quotient of $\fol_0\circ \fol_C$ is isometric to an orbifold of curvature $4$. In particular there is a large variety of non polar singular Riemannian foliations whose leaf space is an orbifold of constant curvature $4$, of any dimension. This should be compared with a recent result of C. Gorodski and A. Lytchak \cite{GL3}, who show that if the quotient $X$ of a non-polar homogeneous foliation on a sphere is isometric to an orbifold, then it is either a weigthed projective space (complex or quaternionic) or it is a good orbifold of curvature $4$ and dimension lower than $5$. In particular, there is only a finite number of orbifolds of curvature 4 what arise as quotients of homogeneous foliations.

\subsection{Strata on the leaf space}

Let $(\sphere^n, \fol)$ be a homogeneous foliation, induced by the action of a compact group $G\In \SO(n+1)$. If $\Sigma$ is a minimal stratum of $\fol$ and $x\in \Sigma$ is a singular point with isotropy group $G_x$, the connected component of $\Fix(G_x)$ through $x$ is a totally geodesic sphere $\sphere^k\In \sphere^n$ that projects via $\pi:\sphere^n\to \sphere^n/G$ to the minimal stratum $\pi(\Sigma)$ of $\sphere^n/\fol$ containing $\pi(x)$. Moreover, the group $G'=N(G_x)/G_x$ acts effectively on $\sphere^k$ by isometries, and there is a map $\sphere^k/G'\to \pi(\Sigma)$. If we let $G'_0$ be the identity component of $G'$, $G'_0$ induces a homogeneous singular Riemannian foliation $\fol'$ on $\sphere^k$, and the stratum $\pi(\Sigma)$ is the quotient of the singular Riemannian foliation $(\sphere^k,\fol')$ (cf. \cite{GL3} where this fact is stated in greater generality).

This is no longer true in the case of (non homogeneous) Clifford foliations. In fact given a Clifford foliation $(\sphere^{2l-1},\fol_C)$, the only singular stratum in the quotient is $\sphere_C\simeq \sphere^m$, which in particular is also minimal but it cannot be the quotient of any singular Riemannian foliation $(\sphere^k, \fol')$, unless $m=2, 4, 8$.

\subsection{Highly curved quotients} Given a composed foliation $(\sphere^{2l-1},\fol_0\circ \fol_C)$ where $\fol_0$ is not polar, the quotient $\sphere^{2l-1}/\fol_0\circ \fol_C$ is an Alexandrov space of curvature $\geq 2$, but not with constant curvature. Some of these examples are homogeneous, but it is not known whether there are other homogeneous examples with these curvature properties.

\subsection{Isometric quotients}
It is not hard to produce non congruent homogeneous foliations $(\sphere^n,\fol)$, $(\sphere^{n'},\fol')$ with isometric quotients. Such foliations have been recently been extensively studied by Gorodski and Lytchak in \cite{GL1, GL2}. However, to the best the author's knowledge there are no known examples of non congruent homogeneous foliations $(\sphere^n,\fol)$, $(\sphere^{n'},\fol')$ that admit an isometry $I:\sphere^n/\fol\to\sphere^{n'}/\fol'$  \emph{which preserves the dimension of the leaves}.

It was shown in  \cite{AR} that if two singular Riemannian foliations $(M,\fol)$, $(M',\fol')$ admit an isometry $I:M/\fol\to M'/\fol'$ which preserves the dimension of the leaves, then the two foliations admit isomorphic sheaves of smooth basic functions (i.e. smooth functions that are constant along the leaves). In the case $M=M'=\sphere^n$ it was hard to come up with non congruent examples with this property.

By leaving the realm of homogeneous foliations and using composed foliations, however, we can produce arbitrary numbers of pairwise non congruent foliations $(\sphere^{n_i},\fol^i)$ all of whose quotients are isometric, and the corresponding leaves have the same dimension.
In fact, fixing an integer $r$, consider $r$ geometrically inequivalent Clifford systems $C^{(i)}=\left(P^{(i)}_0,\ldots, P^{(i)}_m\right)$ on $\RR^{2l}$, with $i=1,\ldots, r$. By the classification of Clifford systems, such $C^{(i)}$ exist if $m$ is a multiple of $4$ and $l=k\delta(m)$ for some $k\geq 2r+2$. By Proposition \ref{P:distinguish} the foliations $\fol_{C^{(i)}}$ on $\sphere^{2\delta(m)-1}$ are not congruent but the quotients $\sphere^{2\delta(m)-1}/\fol_{C^{(i)}}$, $i=1,\ldots, r$ are all isometric to each other, with corresponding leaves of the same dimension. Given a singular Riemannian foliation $(\sphere^{m},\fol_0)$, the foliations $\fol_0\circ \fol_{C^{(i)}}$ are also not congruent, but again the quotients $\sphere^{2\delta(m)-1}/(\fol_0\circ\fol_{C^{(i)}})$ are all isometric to each other, 
with corresponding leaves of the same dimension.



\bibliographystyle{amsplain}

\end{document}